\documentclass[12pt]{amsart}
\usepackage{amsmath}
\usepackage{amscd}
\usepackage{amssymb}
\usepackage[dvips]{graphicx}
\usepackage[margin=3.8cm]{geometry}

\newtheorem{theorem}[subsection]{Theorem}
\newtheorem{lemma}[subsection]{Lemma}

\newtheorem{corollary}[subsection]{Corollary}
\newtheorem{remark}[subsection]{Remark}
\newtheorem{example}[subsection]{Example}
\newtheorem{definition}[subsection]{Definition}
\newtheorem{claim}[subsection]{Claim}

\makeatletter
\@addtoreset{equation}{section}
\@addtoreset{figure}{section}
\@addtoreset{table}{section}
\makeatother

\newcommand\CCC{{\mathbb C}}
\newcommand\NNN{{\mathbb N}}

\newcommand\RRR{{\mathbb R}}

\newcommand\ZZZ{{\mathbb Z}}

\newcommand\id{\mathrm{id}}

\newcommand\im{\mathrm{im}}
\newcommand\Int{\mathrm{Int}}
\newcommand\GL{\mathrm{GL}}

\newcommand\rank{rank} 

\newcommand\Stab{\mathcal{S}}
\newcommand\Diff{\mathcal{D}}
\newcommand\End{\mathcal{E}}

\newcommand\man{M}
\newcommand\nan{N}
\newcommand\van{V}

\newcommand\partit{\mathcal{P}}
\newcommand\sing{\Sigma}
\newcommand\struc{\Theta}

\newcommand\amap{f}
\newcommand\afunc{\alpha}
\newcommand\Fld{F}
\newcommand\hFld{G}

\newcommand\flow{\Phi}
\newcommand\Fol{\mathcal{F}}

\newcommand\partitf{\partit_{\amap}}
\newcommand\singf{\sing_{\amap}}
\newcommand\strucf{\struc_{\amap}}

\newcommand\partitFld{\partit_{\Fld}}
\newcommand\singFld{\sing_{\Fld}}
\newcommand\strucFld{\struc_{\Fld}}

\newcommand\partithFld{\partit_{\hFld}}
\newcommand\singhFld{\sing_{\hFld}}
\newcommand\struchFld{\struc_{\hFld}}

\newcommand\partitFol{\partit_{\Fol}}
\newcommand\singFol{\sing_{\Fol}}
\newcommand\strucFol{\struc_{\Fol}}

\newcommand\EV{\End(\struc,\van)}

\newcommand\EVFld{\End(\strucFld,\van)}

\newcommand\DV{\Diff(\struc,\van)}

\newcommand\DVFld{\Diff(\strucFld,\van)}

\newcommand\Estr{\End(\struc)}
\newcommand\Dstr{\Diff(\struc)}

\newcommand\Dstrf{\Diff(\strucf)}

\newcommand\EhF{\End(\struchFld)}
\newcommand\DhF{\Diff(\struchFld)}

\newcommand\DF{\Diff(\strucFld)}

\newcommand\EidV[1]{\End_{\id}(\struc,\van)^{#1}}
\newcommand\DidV[1]{\Diff_{\id}(\struc,\van)^{#1}}

\newcommand\EidVFld[1]{\End_{\id}(\strucFld,\van)^{#1}}
\newcommand\DidVFld[1]{\Diff_{\id}(\strucFld,\van)^{#1}}

\newcommand\EidhFld[1]{\End_{\id}(\struchFld)^{#1}}
\newcommand\DidhFld[1]{\Diff_{\id}(\struchFld)^{#1}}

\newcommand\Didf[1]{\Diff_{\id}(\strucf)^{#1}}

\newcommand\NNi{\overline{\NNN}_{0}}
\newcommand\NNz{\NNN_{0}}

\newcommand\Shift{\varphi}

\newcommand\domflow{\mathrm{dom}(\flow)}
\newcommand\funcflow{\mathrm{func}(\flow,\van)}

\newcommand\con{r}
\newcommand\bcon{\widetilde{\con}}
\newcommand\homot{H}

\newcommand\dif{h}
\newcommand\tdif{\widetilde{\dif}}
\newcommand\bhomot{\widetilde{\homot}}
\newcommand\imcon{R}

\newcommand\Stabf{\Stab(\amap)}
\newcommand\Stabfpl{\Stab^{+}(\amap)}
\newcommand\StabIdf[1]{\Stab_{\id}(\amap)^{#1}}

\newcommand\condif{\eta}
\newcommand\rval{\varepsilon}

\newcommand\LStab{\mathcal{LS}}
\newcommand\LStabf{\LStab(\amap)}

\newcommand\GLPR{\GL^{+}(2,\RRR)}

\newcommand\Sym[1]{{\mathrm{Sym}(#1)}}

\begin{document}
\begin{abstract}
Let $f:\mathbb{R}^2\to\mathbb{R}$ be a homogeneous polynomial and $\mathcal{S}(f)$ be the group of diffeomorphisms $h$ of $\mathbb{R}^2$ preserving $f$, i.e. $f \circ h =f$.
Denote by $\mathcal{S}_{\mathrm{id}}(f)^{r}$, $(0\leq r \leq \infty)$, the identity component of $\mathcal{S}(f)$ with respect to the weak Whitney $C^{r}_{W}$-topology.
We prove that $\mathcal{S}_{\mathrm{id}}(f)^{\infty} = \cdots = \mathcal{S}_{\mathrm{id}}(f)^{1}$ for all $f$ and that
$\mathcal{S}_{\mathrm{id}}(f)^{1} \not= \mathcal{S}_{\mathrm{id}}(f)^{0}$ if and only if $f$ is a product of at least two distinct irreducible over $\mathbb{R}$ quadratic forms.
\end{abstract}
\author{Sergiy Maksymenko}
\title
[Partition preserving diffeomorphisms]
{Connected components of partition preserving diffeomorphisms}
\address{Topology dept., Institute of Mathematics of NAS of Ukraine, Te\-re\-shchen\-kivs'ka st. 3, Kyiv, 01601 Ukraine}
\email{maks@imath.kiev.ua}
\date{30.11.2008}
\thanks{This research is partially supported by Grant of President of Ukraine.}
\maketitle
\section{Introduction}
Let $\amap:\RRR^2\to\RRR$ be a homogeneous polynomial of degree $p\geq 1$.
Thus up to a sign we can write
\begin{equation}\label{equ:homog_poly}
\amap(x,y) = \pm \prod_{i=1}^{l} L_i^{\alpha_i}(x,y) \cdot  \prod_{j=1}^{k} Q_j^{\beta_j}(x,y),
\end{equation}
where every $L_i$ is a linear function, $Q_j$ is a positive definite quadratic form, $\alpha_i, \beta_j \geq 1$, and
$$
\frac{L_i}{L_{i'}} \not= \mathrm{const} \ \text{for $i\not=i'$},
\qquad
\frac{Q_j}{Q_{j'}} \not= \mathrm{const} \ \text{for $j\not=j'$}.
$$

Denote by $\Stabf=\{\dif\in\Diff(\RRR^2) \, : \, \amap\circ \dif = \amap \}$ the stabilizer of $\amap$ with respect to the right action of the group $\Diff(\RRR^2)$ of $C^{\infty}$-diffeomorphisms of $\RRR^2$ on the space $C^{\infty}(\RRR^2,\RRR)$.
It consists of diffeomorphisms of $\RRR^2$ preserving every level-set $\amap^{-1}(c)$ of $\amap$, $(c\in \RRR)$.

Let $\StabIdf{r}$, $(0\leq r \leq \infty)$ be the identity component of $\Stabf$ with respect to weak Whitney $C^{r}_{W}$-topology.
Thus $\StabIdf{r}$ consists of diffeomorphisms $\dif\in\Stabf$ isotopic in $\Stabf$ to $\id_{\RRR^2}$ via (an $\amap$-preserving isotopy) $\homot:\RRR^2\times I\to\RRR^2$ whose partial derivatives in $(x,y)\in\RRR^2$ up to order $r$ continuously depend on $(x,y,t)$, see Section~\ref{sect:Whitney_topologies} for a precise definition.
Then it is easy to see that 
$$ \StabIdf{\infty} \subset \cdots \StabIdf{r} \subset \cdots \subset \StabIdf{1} \subset\StabIdf{0}.$$

It follows from results~\cite{Maks:jets,Maks:hamv2} that $\StabIdf{\infty}=\StabIdf{1}$.
Moreover, it is actually proved in~\cite{Maks:Shifts} that $\StabIdf{\infty}=\StabIdf{0}$ for $p\leq2$, see also~\cite{Maks:AGAG:2006}.
The aim of this note is to prove the following theorem describing the relation between $\StabIdf{r}$ for all $p\geq 1$.
\begin{theorem}\label{th:Stabf1_not_Stabf0}
Let $\amap:\RRR^2\to\RRR$ be a homogeneous polynomial of degree $p\geq 1$.
Then $\StabIdf{\infty}=\cdots=\StabIdf{1}$.
Moreover, $\StabIdf{1}\not=\StabIdf{0}$ if and only if $\amap$ is a product of at least two distinct definite quadratic forms, i.e. $\amap = Q^{\beta_1}_1\cdots Q^{\beta_k}_k$ for $k\geq2$.
\end{theorem}
This theorem is based on a rather general result about partition preserving diffeomorphisms, see Theorem~\ref{th:main}.
The applications of Theorem~\ref{th:Stabf1_not_Stabf0} will be given in another paper concerning smooth functions on surfaces with isolated singularities.

{\bf Structure of the paper.}
In Section~\ref{sect:Whitney_topologies} we describe homotopies which induce continuous paths into functional spaces with weak Whitney $C^{r}_{W}$-topologies.
Section~\ref{sect:sing-partit} introduces the so called \emph{singular partitions} of manifolds being the main object of the paper.
Section~\ref{sect:inv-contr} contains the main result, Theorem~\ref{th:main}, about invariant contractions of singular partitions.
In Section~\ref{sect:appl:stabs} an application of this theorem to local extremes of smooth functions is given.
Section~\ref{sect:lin_sym} contains a description of the group of linear symmetries of $\amap$.
Finally in Section~\ref{sect:proof-th:Stabf1_not_Stabf0} we prove Theorem~\ref{th:Stabf1_not_Stabf0}.

\section{$r$-homotopies}\label{sect:Whitney_topologies}
Denote $\NNz = \NNN\cup\{0\}$ and $\NNi =\NNN \cup \{0,\infty\}$.

Let $\man,\nan$ be two smooth manifolds of dimensions $m$ and $n$ respectively.
Then for every $r\in\NNi$ the space $C^{r}(\man,\nan)$ admits the so called \emph{weak} Whitney topology denoted by $C^{r}_{W}$, see e.g.~\cite{Leslie:Topology,Hirsch:DiffTop}.

Recall, e.g.~\cite[\S44.IV]{Kuratowski:2} that there exists a homeomorphism $$C^{0}(I,C^{0}(\man,\nan)) \approx C^{0}(\man\times I,\nan)$$ with respect to the corresponding $C^{0}_{W}$-topologies (also called \emph{compact open} ones) associating to every (continuous) path $w:I \to C^{0}(\man,\nan)$ a homotopy $\homot:\man\times I\to \nan$ defined by $\homot(x,t) = w(t)(x)$.

We will now describe homotopies inducing continuous paths $w:I \to C^{r}(\man,\nan)$ with respect to $C^{r}_{W}$-topologies.

\begin{definition}
Let $\homot:\man\times I\to\nan$ be a homotopy and $r\in\NNi$.
We say that $\homot$ is an {\bfseries $r$-homotopy} if
\begin{enumerate}
\item
$\homot_t:\man\to\nan$ is $C^{r}$ for every $t\in I$;
\item
partial derivatives of $\homot(x,t)$ by $x$ up to order $r$ continuously depend on $(x,t)$.
\end{enumerate}
More precisely, let $z\in\man\times I$.
Then in some local coordinates at $z$ we can regard $\homot$ as a map 
$$\homot=(\homot_1,\ldots,\homot_n):\RRR^m\times I\to\RRR^{n}$$
such that for every fixed $t$ and $i$ the function $\homot_i(x,t)$ is $C^{r}$.
Condition {\rm(2)} requires that for every $i=1,\ldots,n$ and every non-negative integer vector $k=(k_1,\ldots,k_m)$ of norm $|k|=\sum_{j=1}^{m} k_i \leq r$ the function 
$$
\frac{\partial^{|k|} \homot_i}{\partial x_1^{k_1} \,\cdots \,\partial x_m^{k_m}} (x_1,\ldots,x_m,t) 
$$
continuously depend on $(x,t)$.
\end{definition}

Thus a $0$-homotopy $\homot$ is just a usual homotopy.

It easily follows from the definition of $C^{r}_{W}$-topologies that a path $w:I \to C^{r}(\man,\nan)$ is continuous from the standard topology of $I$ to $C^{r}_{W}$-topology of $C^{r}(\man,\nan)$ if and only if the corresponding homotopy $\homot:\man\times I\to\nan$ is an $r$-homotopy.

We can also define a \emph{$C^{r}$-homotopy} as a $C^{r}$-map $\man\times I\to\nan$.
Evidently, every $C^{r}$-homotopy is an $r$-homotopy as well, but the converse statement is not true.

\begin{example}
Let $\homot:\RRR\times I\to\RRR$ be given by 
$$\homot(x,t) =
\left\{\begin{array}{cl} 
t \ln(x^2+t^2), & (x,t)\not=(0,0), \\
0, & (x,t)=(0,0).
\end{array}
\right.
$$
Then $\homot$ is continuous, while $\frac{\partial \homot}{\partial x}=\frac{2tx}{x^2+y^2}$ is $C^{\infty}$ for every fixed $t$ as a function in $x$ but discontinuous at $(0,0)$ as a function in $(x,t)$. 
In other words $\homot$ is a $0$-homotopy but not a $1$-homotopy.

Moreover, define $G:\RRR\times I\to\RRR$ by $G(x,t)=\int_{0}^{x} H(y,t)dy$.
Then $G$ a $1$-homotopy but not a $2$-homotopy.
\end{example}

\section{Singular partitions of manifolds}\label{sect:sing-partit}
Let $\man$ be a smooth manifold equipped with a partition $\partit=\{ \omega_i\}_{i\in\Lambda}$, i.e. a family of subsets $\omega_i$ such that 
$$
\man=\mathop\cup\limits_{i\in\Lambda} \omega_i,    \qquad \omega_i\cap\omega_j=\varnothing\ (i\not=j).
$$
In general $\Lambda$ may be even uncountable and $\omega_i$ are not necessarily closed in $\man$.
Let also $\Lambda'$ be a (possibly empty) subset of $\Lambda$ and $\sing=\{\omega_i\}_{i\in\Lambda'}$ be a subfamily of $\partit$ thought as a set of ``singular'' elements.
Then the pair $\struc=(\partit,\sing)$ will be called a \emph{singular partition} of $\man$.

\begin{example}\label{exmp:part-Fld}\rm
Let $\Fld$ be a vector field on $\man$, $\partitFld$ be the set of orbits of $\Fld$, and $\singFld$ be the set of singular points of $\Fld$.
Then the pair $\strucFld=(\partitFld,\singFld)$ will be called the \emph{singular partition of $\Fld$}.
\end{example}

\begin{example}\label{exmp:part-func}\rm
Let $\amap:\RRR^m\to\RRR^n$ be a smooth map, $x\in\RRR^m$ be a point, and $J(\amap,x)$ be the Jacobi matrix of $\amap$ at $x$.
Then $x\in\RRR^m$ is called \emph{critical} for $\amap$ if $\rank\, J(\amap,x) < \min\{m,n\}$.
Otherwise $x$ is \emph{regular}.
This definition naturally extends to maps between manifolds.

Let $\man$, $\nan$ be smooth manifolds and $\amap:\man\to\nan$ a smooth map.
Denote by $\singf$ the set of critical points of $\amap$.
Consider the following partition $\partitf$ of $\man$: a subset $\omega \subset\man$ belongs to $\partitf$ iff $\omega$ is either a critical point of $\amap$, or a connected component of the set of the form $\amap^{-1}(y) \setminus \singf$ for some $y\in\nan$.
Then the pair $\strucf=(\partitf,\singf)$ will be called the \emph{singular partition of $\amap$}.
Evidently, every $\omega\in\partitf\setminus\singf$ is a submanifold of $\man$.
\end{example}

\begin{example}\label{exmp:part-func-Fld}\rm 
Assume that in Example~\ref{exmp:part-func} $\dim\man = \dim\nan+1$ and both $\man$ and $\nan$ are orientable.
Then every element of $\partitf\setminus\singf$ is one-dimensional and orientations of $\man$ and $\nan$ allow to \emph{coherently orient all the elements of $\partitf\setminus\singf$}.
Moreover, it is even possible to construct a vector field $\Fld$ on $\man$ such that the singular partitions $\strucf$ and $\strucFld$ coincide.

In particular, let $\man$ be an orientable surface and $\amap:\man\to\RRR$ be a smooth function.
Then $\man$ admits a symplectic structure, and in this case we can assume that $\Fld$ is the corresponding Hamiltonian vector field of $\amap$.
\end{example}

\begin{example}\label{exmp:part-Fol}\rm
Let $\Fol$ be a foliation on $\man$ with singular leaves, $\partit$ be the set of leaves of $\Fol$, and $\sing$ be the set of its singular leaves (having non-maximal dimension).
Then the pair $\strucFol=(\partitFol,\singFol)$ will be called the \emph{singular partition of $\Fol$}.
This example generalizes all previous ones.
\end{example}

Let $\struc=(\partit,\sing)$ be a singular partition on $\man$.
For every open subset $\van\subset\man$ denote by $\EV$ the subset of $C^{\infty}(\van,\man)$ consisting of maps $\amap:\van\to\man$ such that 
\begin{enumerate}
\item
$\amap(\omega_i\cap\van)\subset\omega_i$ for all $\omega_i\in\partit$ and 
\item
$\amap$ is a local diffeomorphism at every point $z$ belonging to some singular element $\omega\in\sing$.
\end{enumerate} 

Let also $\DV$ be the subset of $\EV$ consisting of \emph{immersions}, i.e. \emph{local diffeomorphisms}.
For $\van=\man$ we abbreviate 
$$\Estr=\End(\struc,\man),\qquad \Dstr=\Diff(\struc,\man).$$

For every $r\in\NNi$ denote by $\EidV{r}$, resp. $\DidV{r}$, the path-component of the identity inclusion $i_{\van}:\van\subset\man$ in $\EV$, resp. in $\DV$, with respect to the induced $C^{r}_{W}$-topology, see Section~\ref{sect:Whitney_topologies}. 

Thus $\EidV{r}$ (resp. $\DidV{r}$) consists of maps (resp. immersions) $\van\subset\man$ which are $r$-homotopic ($r$-isotopic) to $i_{\van}:\van\subset\man$ in $\EV$ (resp. in $\DV$).

Evidently,
\begin{equation}\label{equ:Eid_Inclusions}
\EidV{\infty} \subset \cdots  \subset \EidV{1} \subset \EidV{0},
\end{equation}
and similar relations hold for $\DidV{r}$.

The following notion turns out to be useful for studying singular partitions of vector fields.

\subsection{Shift-map of a vector field.}\label{sect:shift-map}
Let $\Fld$ be a vector field on $\man$ and $$\flow:\man\times\RRR \supset \domflow \to\man$$ be the local flow of $\Fld$ defined on some open neighbourhood $\domflow$ of $\man\times 0$ in $\man\times \RRR$.
For every open subset $\van\subset\man$ let also 
$$\funcflow=\{\afunc\in C^{\infty}(\van,\RRR) \ : \ \Gamma_{\afunc}\subset \domflow \},$$
where $\Gamma_{\afunc} = \{ (x,\afunc(x)) \ : \ x\in\van\} \subset \man\times\RRR$ is the graph of $\afunc$.
Then $\funcflow$ is the largest subset of $C^{\infty}(\van,\RRR)$ on which the following \emph{shift-map} is defined:
$$
\Shift_{\van}: \funcflow \to C^{\infty}(\van,\man), \qquad 
\Shift_{\van}(\afunc)(x) = \flow(x,\afunc(x)),
$$
for $\afunc\in \funcflow$, $x\in\van$.

\begin{lemma}\label{lm:imShift_EidVFinfty}
Let $\strucFld$ be the singular partition of $\man$ by orbits of $\Fld$.
Then 
\begin{equation}\label{equ:imSh_EidFinf}
\im(\Shift_{\van}) \subset \EidVFld{\infty}.
\end{equation}
Moreover, if $\DidVFld{r}\subset\im(\Shift)$ for some $r\in\NNi$, then 
$$\DidVFld{\infty}=\cdots=\DidVFld{r+1}=\DidVFld{r}.$$
\end{lemma}
\begin{proof}
Let $\afunc\in\funcflow$ and $\amap=\Shift(\afunc)$, i.e. $\amap(x)=\flow(x,\afunc(x))$.
Then $\amap(\omega\cap\van)\subset\omega$ for every orbit of $\Fld$.
Moreover by~\cite[Lemma~20]{Maks:Shifts} $\amap$ is a local diffeomorphism at a point $x\in\van$ iff $d\afunc(\Fld)(x)\not=-1$, where $d\afunc(\Fld)(x)$ is the Lie derivative of $\afunc$ along $\Fld$ at $x$.
Hence $\amap$ is so at every singular point $z$ of $\Fld$, since $d\afunc(\Fld)(z)=0\not=-1$, \cite[Corollary~21]{Maks:Shifts}.
Therefore $\amap\in\EVFld$.
Moreover an $\infty$-homotopy of $\amap$ to $i_{\van}:\van\subset\man$ in $\EVFld$ can be given by $\amap_t(x)=\flow(x,t\afunc(x))$.
Thus $\amap\in\EidVFld{\infty}$.

Finally, suppose that $\amap\in\DidVFld{r}$.
Then the restriction of $\amap$ to any non-constant orbit $\omega$ of $\Fld$ is an orientation preserving local diffeomorphism.
Therefore $d\afunc(\Fld)(z)>-1$ on all of $\van$.
Hence $d(t\afunc)(\Fld)(z)>-1$ for all $t\in I$ as well, i.e. $\amap_t\in\DVFld$.
This implies that $\amap\in\DidVFld{\infty}$.
\end{proof}

\begin{example}\label{exmp:lin-v-f}\rm
Let $A$ be a real non-zero $(m\times m)$-matrix, $\Fld(x)=Ax$ be the corresponding linear vector field on $\RRR^m$, and $\van$ be a neighbourhood of the origin $0$.
Then the shift-map $\Shift_{\van}$ is given by 
$$\Shift(\afunc)(x) = \flow(x,\afunc(x))=e^{A\afunc(x)}x.$$
It is shown in~\cite{Maks:Shifts} that in this case $\im(\Shift_{\van})=\EidVFld{0}$.
Hence for all $r\in\NNi$ we have
$$\im(\Shift_{\van}) = \EidVFld{\infty} = \cdots =\EidVFld{r},$$
$$\DidVFld{\infty} = \cdots =\DidVFld{r}.$$
\end{example}

\section{Invariant contractions}\label{sect:inv-contr}
Let $\struc=(\partit,\sing)$ be a singular partition on a manifold $\man$.
We will say that a subset $\van\subset\man$ is \emph{$\struc$-invariant}, if it consists of full elements of $\struc$, i.e. if $\omega\in\partit$ and $\omega\cap\van\not=\varnothing$, then $\omega\subset\van$.

\begin{definition}
Let $Z\subset \man$ be a closed subset such that every point $z\in Z$ is a singular element of $\struc$, i.e. $\{z\}\in\sing$.
Say that $\struc$ has an {\bfseries invariant $r$-contraction to $Z$} if there exists a closed $\struc$-invariant neighbourhood $\van$ of $Z$ being a smooth submanifold of $\man$ and a homotopy $\con:\van\times I\to\van$ such that:
\begin{enumerate}
\item[\rm(i)]
$\con_1 = \id_{\van}$;
\item[\rm(ii)]
$\con_0$ is a proper retraction of $\van$ to $Z$, i.e. $\con_0(\van)=Z$, $\con_0(z)=z$ for $z\in Z$, and $\con_0^{-1}(K)$ is compact for every compact $K\subset Z$;
\item[\rm(iii)]
for every $t\in(0,1]$ the map $\con_t$ is a closed $C^{r}$-embedding of $\van$ into $\van$ such that 
for each $\omega\in\partit$ (resp. $\omega\in\sing$) its image $\con_t(\omega)$ is also an element of $\partit$ (resp. $\sing$);
\item[\rm(iv)]
for each $z\in Z$ the set $\van_z = \con_0^{-1}(z)$ is $\struc$-invariant, and $\con_t(\van_z) \subset \van_z$ for all $t\in I$.
\end{enumerate}
\end{definition}
Since $\van$ is $\struc$-invariant, it follows from (iii) that so is its image $\con_t(\van)$.

\begin{example}\label{exmp:std-sum-sq}\rm
Define $\amap:\RRR^m\to\RRR$ by $\amap(x_1,\ldots,x_m) = \sum_{i=1}^{m} x_i^2$.
Evidently, the singular partition $\strucf$ consists of the origin $0$ and concentric spheres centered at $0$.
For every $s>0$ let $\van_{s}=\amap^{-1}[0,s^2]$ be a closed $m$-disk of radius $s$.
Then $\van_s$ is $\strucf$-invariant and its invariant contraction of $\van_{s}$ to $Z=\{0\}$ can be given by $\con(x,t)=tx$.
\end{example}

\begin{example}\rm
The previous example can be parameterized as follows.
Let $p:\man\to Z$ be an $m$-dimensional vector bundle over a connected, smooth manifold $Z$.
We will identify $Z$ with the image $Z\subset \man$ of the corresponding zero-section of $p$.
Suppose that we are given a norm $\|\cdot\|$ on fibers such that the following function $\amap:\man\to\RRR$ is smooth:
$$
\amap(\xi,z)= \|\xi\|^2,\qquad z\in Z, \ \xi\in p^{-1}(z).
$$
Define the following singular partition $\struc=(\partit,\sing)$ on $\man$, where $\partit$ consists of subsets  $\omega_{s,z}=\amap^{-1}(s)\cap p^{-1}(z)$ for $s\geq0$ and $z\in Z$,
and $\sing=\{ \omega_{0,z} = \{z\} \ : \ z\in Z \}$ consists of points of $Z$.
Thus every fiber $p^{-1}(z)$ is $\struc$-invariant, and the restriction of $\struc$ to $p^{-1}(z)$ is the same as the one in Example~\ref{exmp:std-sum-sq}.

Fix $s>0$ and put $\van=\amap^{-1}[0,s]$.
Then $\van$ is $\struc$-invariant and a $\struc$-invariant contraction of $\van$ to $Z$ can be given by $\con(\xi,z,t)=(t\xi,z)$.
\end{example}

We will now generalize these examples.
Let $\amap:\RRR^m\to\RRR$ be a smooth function and suppose that there exists a neighbourhood $\van$ of $0$ and smooth functions $\alpha_1,\ldots,\alpha_m:\van\to\RRR$ such that 
\begin{equation}\label{equ:Jcond}
\amap = \alpha_1 \cdot \amap'_{x_1} + \cdots + \alpha_m \cdot \amap'_{x_m}.
\end{equation}
Equivalently, let $\Delta(\amap,0)$ be the \emph{Jacobi} ideal of $\amap$ in $C^{\infty}(\van,\RRR)$ generated by partial derivatives of $\amap$.
Then~\eqref{equ:Jcond} means that $\amap\in \Delta(\amap,0)$.

For instance, let $\amap$ be \emph{quasi-homogeneous of degree $d$ with weights $s_1,\ldots,s_m$}, i.e. $\amap(t^{s_1} x_1,\ldots,t^{s_m} x_m) = t^{d} \amap(x_1,\ldots,x_m)$ for $t>0$,  see e.g~\cite[\S12]{AVG:Sing-1}.
Equivalently, we may require that the function $$g(x_1,\ldots,x_m)=\amap(x_1^{s_1},\ldots,x_{m}^{s_m})$$ is homogeneous of degree $d$.
Then the following \emph{Euler identity} holds true:
$$
\amap = \frac{x_1}{s_1} \, \amap'_{x_1} + \; \cdots \;  + \frac{x_1}{s_m} \, \amap'_{x_m}.
$$
In particular, $\amap$ satisfies~\eqref{equ:Jcond}.
Moreover in the complex analytical case\footnote{\,I would like to thank V.\;A.\;Vasilyev for referring me to the paper~\cite{Saito} by K.~Saito.} the identity~\eqref{equ:Jcond} characterizes quasi-homogeneous functions, see~\cite{Saito}.

\begin{lemma}\label{lm:cond_Jid}
Let $\amap:\man\to\RRR$ be a smooth function and $z\in\man$ be an isolated local minimum of $\amap$.
Suppose that $\amap$ satisfies condition~\eqref{equ:Jcond} at $z$, i.e. $\amap\in	\Delta(\amap,z)$.
Then the singular partition $\strucf$ admits an invariant contraction to $z$.
\end{lemma}
\begin{proof}
Since the situation is local, we may assume that $\man=\RRR^{m}$, $z=0$ is a unique critical point of $\amap$ being its global minimum,  $\amap(0)=0$, and  there exists an $\rval>0$ such that $\van=\amap^{-1}[0,\rval]$ is a smooth compact $m$-dimensional manifold with boundary $L=\amap^{-1}(\rval)$.

First we give a precise description of the partition $\partitf$ on $\van$.
Let $\Fld$ be any gradient like vector field on $\van$ for $\amap$, i.e. $d\amap(\Fld)(x)>0$ for $x\not=0$.
Then following~\cite[Th.~3.1]{Milnor:MorseTh} we can construct a diffeomorphism
$$\condif: \van\setminus\{0\} \to L\times (0,\rval]$$
such that $\amap\circ\condif^{-1}(y,t) = t$ for all $(y,t)\in  L\times (0,\rval]$, see Figure~\ref{fig:cone}.

Let  $CL=L\times[0,\rval]/ \{L\times0\}$ be the cone over $L$ and $L_t=\amap^{-1}(t)$ for $t\in(0,\varepsilon]$.
Since $\mathrm{diam}(L_t) \to 0$ when $t\to0$, we obtain that $\condif$ extends to a homeomorphism $\condif:\van\to CL$ by $\condif(0)=\{L\times0\}$.

For every $t\in(0,\rval]$ put $L_t=\amap^{-1}(t)$.
Then $\condif$ diffeomorphically maps $L_{t}$ onto $L\times \{t\}$.

\begin{figure}[ht]
\includegraphics[height=2.5cm]{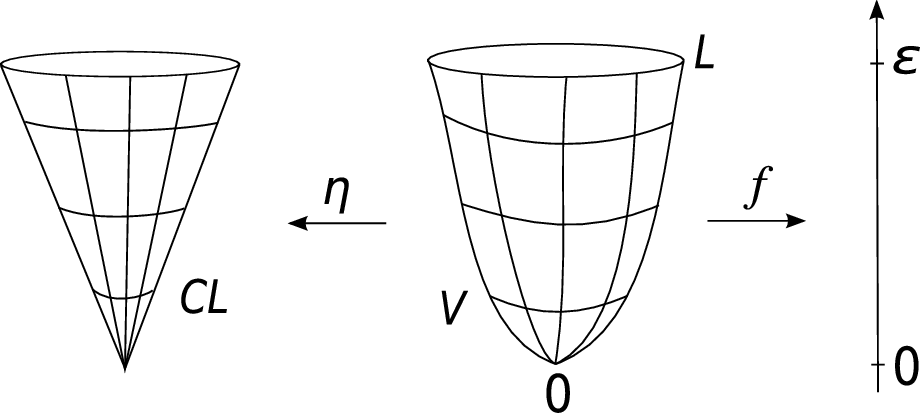}
\caption{}\protect\label{fig:cone}
\end{figure}

\begin{lemma}\label{lm:L_homeq_Sm1}
$L$ is homotopy equivalent to $S^{m-1}$.
\end{lemma}
We will prove this lemma below.
Then it will follow from the generalized Poincar\'e conjecture that $L$ is homeomorphic with the sphere $S^{m-1}$, and even diffeomorphic to $S^{m-1}$ for $m\not=4$.
For $k=1,2$ this statement is rather elementary, for $k=3$ this follows from a recent work of G.~Perelman~\cite{Perelman1,Perelman2}, for $k=4$ from M.~Freedman~\cite{Freedman}, and for $k\geq5$ from S.~Smale~\cite{Smale:PC}, see also~\cite{Milnor:hcob}.

In particular, every $L_t$ is connected, whence the partition $\partitf$ on $\van$ consists of a unique singular element $\{0\}\in\singf$ and sets $L_t$, $t\in(0,\rval]$.

Let us recall the definition of $\condif$.
Notice that every orbit of $\Fld$ starts at $0$ and transversely intersect every $L_t$.
For each $x\in\van\setminus\{0\}$ denote by $q(x)$ a unique point of the intersection of the orbit of $x$ with $L=L_{\rval}=\partial\van$.
Then $\condif:\van\setminus\{0\}\to L\times (0,\rval]$ can be given by the following formula:
$$
\condif(x) = (q(x),\amap(x)).
$$

Also notice that if $\phi:[0,\rval]\to [0,\rval]$ is a (not necessarily surjective) $C^{\infty}$ embedding such that $\phi(0)=0$, then we can define the embedding 
$$w_{\phi}: CL \to CL, \qquad w_{\phi}(y,t) = (y, \phi(t))$$
and therefore the embedding $\con_{\phi}=\condif\circ w_{\phi} \circ \condif^{-1}:\van\to\van$.
Then $\con_{\phi}$ is $C^{\infty}$ on $\van\setminus0$ and diffeomorphically maps $L_t$ onto $L_{q(t)}$ for all $t\in(0,\rval]$.
Moreover $\con_{\phi}(0)=0$, but in general $\con_{\phi}$ is not even smooth at $0$.

Suppose now that $\amap\in\Delta(\amap,0)$, i.e.\! we have a presentation~\eqref{equ:Jcond}.
Consider the following vector field 
$$\Fld=\alpha_1\,\frac{\partial}{\partial x_1} \;+\; \cdots \;+\; \alpha_{m}\,\frac{\partial}{\partial x_m}.$$
Then~\eqref{equ:Jcond} means that $\amap = d\amap(\Fld)$.
Since $\amap(x)>0$ for $x\not=0$, it follows that $\Fld$ is a gradient like vector field for $\amap$.
Therefore we can construct a homeomorphism $\condif:\van\to CL$ using $\Fld$ as above.
It follows from~\cite{Maks:BSM:2006} that in our case this $\condif$ has the following feature: 
\begin{itemize}
 \item 
if $\phi:[0,\rval]\to [0,\rval]$ is a $C^{\infty}$ embedding such that $\phi(0)=0$, then the corresponding embedding $\con_{\phi}:\van\to\van$ is a \emph{diffeomorphism onto its image}.
Moreover, if $\phi_{s}$, $(s\in I)$, is a $C^{\infty}$ isotopy, then so is $\con_{\phi_s}:\van\to\van$.
\end{itemize}

In particular, consider the following homotopy 
$$\phi:[0,\rval]\times I \to [0,\rval], \qquad \phi(t,s)=t(1-s)$$
which contracts $[0,\rval]$ to a point and being an isotopy for $t>0$.
Then the induced homotopy $\con:\van\times I\to \van$ is an invariant $\infty$-contraction of $\strucf$ to $0$.
\end{proof}

\begin{proof}[Proof of Lemma~\ref{lm:L_homeq_Sm1}.]
It suffices to establish that 
\begin{equation}\label{equ:pikL}
\pi_{k}L \approx \pi_{k}S^{m-1}=
\left\{
\begin{array}{cl}
 0, & k=0,\ldots,m-2, \\
\ZZZ, & k=m-1.
\end{array}
\right.
\end{equation}
Then the generator $\mu:S^{m-1}\to L$ of $\pi_{m-1}L \approx \ZZZ$ will yield isomorphisms of the homotopy groups $\pi_{k}S^{m-1}\approx\pi_{k}L$ for all $k\leq m-1 = \dim S^{m-1}=\dim L$.
Now by the well-known Whitehead's theorem $\mu$ will be a homotopy equivalence between $S^{m-1}$ and $L$.

For the calculation of homotopy groups of $L$ consider the exact sequence of homotopy groups of the pair $(\van,L)$:
\begin{equation}\label{equ:exseq:VL}
\cdots \to \pi_{k+1}(\van,L) \to \pi_{k} L \to \pi_{k}\van \to \cdots
\end{equation}
Since $\van$ is homeomorphic with the cone $CL$, $\van$ is contractible, whence $\pi_k\van=0$ for all $k\geq0$.

Moreover, $\pi_{k+1}(\van,L)=0$ for all $k=0,\ldots,m-2$.
Indeed, let $$\xi:(D^{k+1},S^{k}) \to (\van,L)$$ be a continuous map.
We have to show that $\xi$ is  homotopic (as a map of pairs) to a map into $L$.
Since $k+1\leq m-1<\dim\van$, $\xi$ is homotopic to a map into $\van\setminus\{0\}$.
But $L$ is a deformation retract of $\van\setminus\{0\}$, therefore $\xi$ is homotopic to a map into $L$.

Now it follows from~\eqref{equ:exseq:VL} that $\pi_{k} L=0$ for $k=0,\ldots,m-2$.
Hence from Hurewicz's theorem we obtain that $\pi_{m-1}L \approx H_{m-1}(L,\ZZZ)$.
It remains to note that $L$ is a connected closed orientable $(m-1)$-manifold, whence (by Poincar\'e duality) $H_{m-1}(L,\ZZZ) \approx \ZZZ$.
This proves~\eqref{equ:pikL}.
\end{proof}

\begin{remark}\label{rem:quasihom}\rm
If $\amap$ is a quasi-homogeneous function of degree $d$ with weights $s_1,\ldots,s_m$, then we can define an invariant contraction by 
$$ \con(x_1,\ldots,x_m,t) = (t^{s_1} x_1,\ldots,t^{s_m} x_m).$$
If $\amap$ is \emph{homogeneous}, then we can even put $\con(x,t)=tx$.
\end{remark}

\begin{theorem}\label{th:main}
Let $\struc=(\partit,\sing)$ be a singular partition on a manifold $\man$, and $Z\subset\man$ be a closed subset such that every $z\in Z$ is a singular element of $\partit$, i.e. $\{z\}\in\sing$.
Suppose that $\struc$ has an invariant $\infty$-contraction to $Z$ defined on a $\struc$-invariant neighbourhood $\van$ of $Z$.
Let also $\dif\in\EV$ be a map fixed outside some neighbourhood $U$ of $z$ such that $\overline{U}\subset \Int\van$.
Then $\dif\in \EidV{0}$, though $\dif$ not necessarily belongs to $\EidV{r}$ for some $r\geq1$.
\end{theorem}

\begin{remark}\rm
Let $\strucf$ be the singular partition of $\amap(x) = \|x\|^2$ as in Example~\ref{exmp:std-sum-sq}, $\van$ be the unit $n$-disk centered at $0$ and $\con:\van\times I\to\van$ be the invariant contraction of $\strucf$ to a point $Z=\{0\}$ defined by $\con(x,t)=tx$.
Then a $0$-homotopy between $\dif$ and $\id_{\van}$ can be defined by 
$$
\homot_t(x)
\left\{
\begin{array}{cl}
t \dif(\frac{x}{t}), &  t>0, \\ [1.5mm]
0, & t=0.
\end{array}
\right.
$$
c.f.~\cite[Ch.4, Theorems 5.3 \& 6.7]{Hirsch:DiffTop}.
Theorem~\ref{th:main} generalizes this example.
\end{remark}

\begin{proof}[Proof of Theorem~\ref{th:main}.]
Let $\con:\van\times I\to\van$ be an invariant $\infty$-contraction of $\strucFld$ to $Z$.
Define the following map $\homot:\van\times I\to\man$ by
$$
\homot(x,t) = 
\left\{
\begin{array}{cl}
\con_t \circ \dif \circ \con_t^{-1}(x), & \text{if} \ t>0 \ \text{and} \ x\in \con_t(\van), \\ [1.5mm]
x, & \text{otherwise}.
\end{array}
\right.
$$
We claim that $\homot$ is a $0$-homotopy (i.e. just a homotopy) between $\dif$ and the identity inclusion $i_{\van}:\van\subset\man$ in $\EV$.
To make this more obvious we rewrite the formulas for $\homot$ in another way.

The homotopy $\con$ can be regarded as the composition 
$$
\con = p_1 \circ \bcon: \;\van\times I\; \xrightarrow{~~\bcon~~} \;\van\times I\; \xrightarrow{~~p_1~~} \;\van,
$$
where $\bcon$ is the following level-preserving map
$$
\bcon:\van\times I\to \van\times I,
\qquad 
\bcon(x,t) = (\con(x,t), t),
$$
and $p_1:\van\times I \to \van$ is the projection to the first coordinate.
\begin{figure}[ht]
\includegraphics[height=3cm]{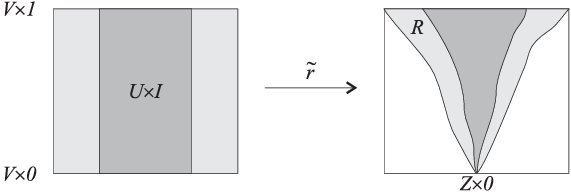}
\caption{}\protect\label{fig:inv-contr}
\end{figure}
It follows from the definition that $\con$ yields a level-preserving embedding $\van\times (0,1]$ to $\van\times I$, see Figure~\ref{fig:inv-contr}.
Denote 
$$\imcon'= \bcon(\van\times (0,1]), \qquad \qquad \imcon = \bcon(\van\times I).$$
Then $\imcon\setminus \imcon'=Z \times 0$.
Define also the following map
$$\tdif:\van\times I\to \van\times I, \qquad \qquad \tdif(x,t)=(\dif(x), t).$$
In these terms, the homotopy $\homot$ is defined by
$$
\homot=p_1\circ\bhomot:\;\van\times I\; \xrightarrow{~~\bhomot~~} \;\van\times I\; \xrightarrow{~~p_1~~} \;\van,$$
where $\bhomot:\van\times I \to \van\times I$ is a level-preserving map given by
$$
\bhomot(x,t) = 
\left\{
\begin{array}{cl}
\bcon \circ \tdif \circ \bcon^{\,\,-1}(x,t), & (x,t)\in \imcon', \\ [1.5mm]
(x,t), & (x,t)\in (\van\times I)\setminus \imcon'.
\end{array}
\right.
$$
Now we can prove that $\homot$ has the desired properties.

Since $\con_1=\id_{\van}$, we have $\homot_1=\dif$.
Moreover $\homot_0=\id_{\van}$.

{\bf 1. Continuity of $\bhomot$ on $\van\times(0,1]$.}
Notice that $\bcon \circ \tdif \circ \bcon^{-1}$ is well-defined and continuous on $R'$.
Moreover, since $\dif$ is fixed on $\van\setminus U$, it follows that $\tdif$ is fixed on $(\van\setminus U)\times I$, whence $\bcon \circ \tdif \circ \bcon^{-1}$ is fixed on the subset $\bcon\bigl( (\van\setminus U)\times (0,1] \bigr) \subset \imcon'$.
This implies that $\bhomot$ is continuous on $\van\times(0,1]$.

{\bf 2. Continuity of $\bhomot$ when $t\to0$.}
Let $z\in\van$. Then $\bhomot(z,0)=(z,0)$.

Suppose that $z\in\van\setminus Z$.
Since $Z$ is closed in $\van$, $\bhomot$ is also fixed and therefore continuous on some neighbourhood of $(z,0)$ in $(\van\times I) \setminus R$.

Let $z\in Z$ and let $W$ be a neighbourhood of $(z,0)$ in $\van\times I$.
We have to find another neighbourhood $W'$ of $(z,0)$ such that $\bhomot(W') \subset W$.

Recall that for every $y\in Z$ we denoted $\van_y = \con_0^{-1}(y)$.
Then $\van_y$ is compact and $\struc$-invariant.

\noindent
{\bf Claim.}
{\em There exist $\varepsilon>0$ and an open neighbourhood $N$ of $z$ in $\van$ such that   $\overline{N}\times[0,\varepsilon]\subset W$ and 
\begin{equation}\label{equ:Vy_0e__V}
\bcon(\van_y \times [0,\varepsilon] ) \subset W, \quad (y\in \overline{N}\cap Z).
\end{equation}}
\noindent{\em Proof.}
Let $N$ be an open neighbourhood of $z$ such that $\overline{N}$ is compact and $\overline{N}\times 0 \subset W$.
Denote
$$
Q \;=\; \con_0^{-1}(\overline{N} \cap Z) \;=\; 
\mathop\cup\limits_{y\in\overline{N}\cap Z} \con_0^{-1}(y) \;=\;
\mathop\cup\limits_{y\in\overline{N}\cap Z} \van_y.
$$
Then $Q$ is a compact subset of $\van$, and $\bcon^{-1}(W)$ is an open neighbourhood of $Q\times 0$ in $\van\times I$.
Hence there exists $\varepsilon>0$ such that $Q\times [0,\varepsilon]\subset \bcon^{-1}(W)$.
This implies~\eqref{equ:Vy_0e__V}.
Decreasing $\varepsilon$ is necessary we can also assume that $\overline{N}\times[0,\varepsilon]\subset W$ as well.
\qed 

Denote $W'= N\times[0,\varepsilon)$.
We claim that $\bhomot(W')\subset W$.

Let $(x,t)\in W'$.
If either $(x,t)\in W'\setminus R'$ or $t=0$, then $\bhomot(x,t)=(x,t)\in W' \subset W$.

Suppose that $(x,t)\in W'\cap R'$. Then $t>0$.
Let also $y=\con_0(x)\in Z$.
Then $\bcon^{\,\,-1}(x,t) \in \van_y\times t$ for all $t\in I$.
Hence
$$
\bcon\circ\tdif\circ\bcon^{\,\,-1}(x,t) 
\;\;\in\;\;
\bcon\circ\tdif(\van_y\times t) 
\;\;\subset\;\;
\bcon(\van_y\times t) 
\;\;\stackrel{\eqref{equ:Vy_0e__V}}{\subset}\;\;
W.
$$
In the second inclusion we have used a $\struc$-invariantness of $\van_y$ and the assumption that $\dif\in\EV$.

{\bf 3. Proof that $\homot_t\in\EV$ for $t\in I$.}
We have to show that (i)~for every $t\in I$ the mapping $\homot_t$ is $C^{\infty}$, (ii)~$\homot_t(\omega)\subset \omega$ for every element $\omega\in\partit$ included in $\van$, and  (iii)~$\homot_t$ is a local diffeomorphism at every point $z$ belonging to some $\omega\in\sing$.

(i) Since $\con_t$, $(t>0)$, is $C^{\infty}$ and $\dif$ is fixed on $\van\setminus U$, it follows that $\homot_t$ is $C^{\infty}$ as well.

(ii) Let $\omega \subset \van$ be an element of $\partit$ (resp. $\sing$).

If $\omega\subset\van\setminus\con_t(\van)$, then $\homot_t$ is fixed on $\omega$, whence $\homot_t(\omega)=\omega\in\partit$  (resp. $\sing$).

Suppose that $\omega\subset\con_t(\van)$.
Since $\con_t(\van)$ is $\struc$-invariant, $\omega=\con_t(\omega')$ for some another element $\omega'\in\partit$  (resp. $\sing$).
Then $\dif(\omega')\subset\omega'$, whence
$$
\homot_t(\omega) \;\;=\;\;
\con_t \circ \dif \circ \con_t^{-1}(\omega) \;\;=\;\;
\con_t \circ \dif (\omega') \;\;\subset\;\;
\con_t (\omega') \;\;=\;\; \omega.
$$

(iii) Suppose that $\omega\in\sing$ and let $x\in\omega$.

If $x \in \van\setminus\con_t(U)$, then $\homot_t$ is fixed in a neighbourhood of $x$, and therefore it is a local diffeomorphism at $x$.

Suppose that $x=\con_t(x')\in\con_t(U)$ for some $x'\in U$ and let $\omega'\in\sing$ be the element containing $x'$.
Then $\dif$ is a local diffeomorphism at $x'$, whence $\homot_t=\con_t \circ \dif \circ \con_t^{-1}$ is a local diffeomorphism at $x$.
\end{proof}

\section{Stabilizers of smooth functions}\label{sect:appl:stabs}
Let $B^m\subset\RRR^m$ be the unit disk centered at the origin $0$, $S^{m-1} = \partial B^m$ be its boundary sphere, $\amap:B^m\to\RRR$ be a $C^{\infty}$ function, and $\strucf$ be the singular partition of $\amap$.

Let $\Stabf=\{\dif\in\Diff(B^m) \, : \, \amap\circ \dif = \amap \}$ be the stabilizer of $\amap$ with respect to the right action of the group $\Diff(B^m)$ of diffeomorphisms of $B^m$ on the space $C^{\infty}(B^m,\RRR)$.
Denote by $\Stabfpl$ the subgroup of $\Stabf$ consisting of orientation preserving diffeomorphisms.
For $r\in\NNi$ let also $\StabIdf{r}$ be the identity component of $\Stabf$ with respect to the $C^{r}$-topology.
Then 
$$
\StabIdf{\infty} \;\subset\cdots\subset\; \StabIdf{r+1} \;\subset\; \StabIdf{r} \;\subset\cdots\subset\;  \StabIdf{0}  \;\subset\; \Stabfpl.
$$

\begin{theorem}\label{th:Stabf-Bn}
Let $m\geq2$, $\amap:B^m\to [0,1]$ be a $C^{\infty}$ function such that $0$ is a unique critical point of $\amap$ being its global minimum, $\amap(0)=0$, and $\amap(S^{m-1})=1$.
Denote by $\Stab$ the subgroup of $\Stabf$ consisting of diffeomorphisms $\dif$ such that $\dif|_{S^{m-1}}:S^{m-1}\to S^{m-1}$ is $C^{\infty}$-isotopic to $\id_{S^{m-1}}$.
Suppose also that the singular partition $\strucf$ of $\amap$ has an invariant $\infty$-contraction to $0$.
Then $\Stab\subset \StabIdf{0}$.

If $m=2,3,4$, then $\Stab=\StabIdf{0}=\Stabfpl$.
\end{theorem}

For the proof we need the following two simple standard statements concerning smoothing homotopies at the beginning and at the end, see e.g.~\cite[pp.74 \& 118]{Pontryagin}  and~\cite[p.~205]{MilnorWallace:DiffTop}.
Let $M$ be a closed smooth manifold.
\begin{claim}\label{clm:1}
Let $a,b,c\in\RRR$ be numbers such that $0<a<b<c$, and $N=M\times(0,c]$.
Then we have a foliation on $N$  by submanifolds $M\times t$, $t\in(0,c]$.
Let also $\dif:N \to N$ be a $C^{\infty}$ leaf preserving diffeomorphism, i.e., $\dif(x,t)=(\phi(x,t),t)$ for some $C^{\infty}$ map $\phi:M \times (0,c]\to M$ such that for every $t\in(0,c]$ the map $\phi_{t}:M \to M$ is a diffeomorphism.
Then there exists a leaf preserving isotopy relatively to $M\times(0,a]$ of $\dif$ to a diffeomorphism $\hat\dif(x,t)=(\hat\phi(x,t),t)$ such that $\hat\phi_t=\phi_c$ for all $t\in[b,c]$.
\end{claim}
\begin{proof}
Let $\mu:(0,c]\to(0,c]$ be a $C^{\infty}$ function such that $\mu(t)=t$ for $t\in(0,a]$ and $\mu(t)=c$ for $t\in[b,c]$, see Figure~\ref{fig:func_mu_nu}a).
Define the following lead preserving isotopy $H:N\times I \to N$ by
$$
H_s(x,t) = (\phi(x, (1-s)t + s\mu(t)), t).
$$
Then it easy to see that $H_0=\id_{N}$, $H_t=\dif$ on $M\times(0,a]$ and $\hat\dif=H_1$ satisfies conditions of our claim.
\end{proof}

\begin{claim}\label{clm:2}
Let $d<e\in(0,1)$ and $G:M\times I\to M$ be a $C^{\infty}$ homotopy (isotopy).
Then there exists another $C^{\infty}$ homotopy (isotopy) $G':M\times I\to M$ such that $\hat G_{t}=G_0$ for $t\in[0,d]$ and $\hat G_{t}=G_1$ for $t\in[e,1]$.
\end{claim}
\begin{proof}
Take any $C ^{\infty}$ function $\nu:I\to I$ such that $\nu[0,d]=0$ and $\nu[e,1]=1$, and put $\hat G_{t}=G_{\nu(t)}$, see Figure~\ref{fig:func_mu_nu}b).
\end{proof}
\begin{center}
\begin{figure}[ht]
\begin{tabular}{ccc}
\includegraphics[height=2.5cm]{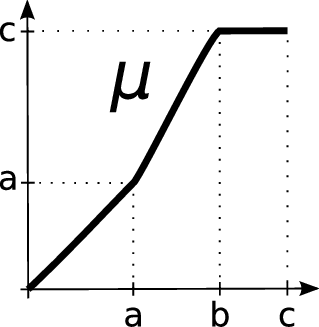}
& \qquad\qquad &
\includegraphics[height=2.5cm]{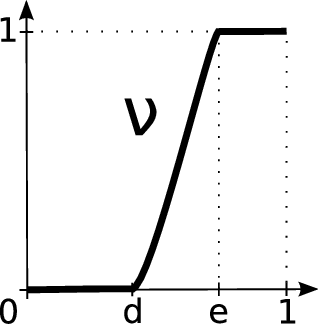} \\
a) & & b)
\end{tabular}
\caption{}\protect\label{fig:func_mu_nu}
\end{figure}
\end{center}

\begin{proof}[Proof of Theorem~\ref{th:Stabf-Bn}.]
Since $0$ is a unique critical point of $\amap$ and $\amap$ is constant on $S^{m-1}$, it follows from the arguments of the proof of Lemma~\ref{lm:cond_Jid} that there exists a diffeomorphism $\condif:D^{m}\setminus0\to S^{m-1}\times (0,1]$ such that $\amap\circ\condif^{-1}(y,t)=t$.
In particular, for every $t\in[0,1]$ the set $\amap^{-1}(t)$ is diffeomorphic with $S^{m-1}$.
Since $S^{m-1}$ is connected, we obtain that $\Dstrf = \Stabf$.
Therefore $\Didf{r} = \StabIdf{r}$ for all $r\in\NNi$.

By assumption there exists an invariant $\infty$-contraction of $\strucf$ to $0$ defined on some $\strucf$-invariant neighbourhood $\van$ of $0$.
Therefore we can assume that $\van=\amap^{-1}[0,2c]$ for some $c\in(0,\frac{1}{2})$.
\begin{lemma}\label{lm:isotopy_h_hpr}
There exists a $C^{\infty}$-isotopy of $\dif$ in $\Stab$ to a diffeomorphism $\hat\dif$ fixed on $\amap^{-1}[c,1]$.
Then is follows from Theorem~\ref{th:main} that $\hat\dif$ and therefore $\dif$ belong to $\Didf{0}=\StabIdf{0}$.
\end{lemma}
\begin{proof}
Since $\dif$ preserves $\amap$, it follows that the following diffeomorphism
$$g=\condif\circ\dif\circ\condif^{-1}:S^{m-1}\times(0,1] \to S^{m-1}\times(0,1]$$
is leaf preserving, i.e. $\dif(S^{m-1}\times t)=S^{m-1}\times t$ for all $t\in(0,1]$.
Then by Claim~\ref{clm:1} we can assume that $g|_{S^{m-1}\times t}=h|_{S^{m-1}}$ for all $t\in[0.5,1]$.

Take any $a\in(0,c)$.
It suffices to find a leaf preserving isotopy relatively to $S^{m-1}\times (0,a]$ of $g$ to a diffeomorphism $\hat g$ fixed on $S^{m-1}\times[c,1]$.
This isotopy will yield an isotopy relatively to $\amap^{-1}[0,a]$ of $\dif$ in $\Stab$ to a diffeomorphism $\hat\dif$ which is fixed on $\amap^{-1}[c,1]$.

By assumptions of our theorem there exists a $C^{\infty}$ isotopy $G:S^{m-1}\times [1,2] \to S^{m-1}$  such that $G_{1}=\dif|_{S^{m-1}}$ and $G_{2}=\id_{S^{m-1}}$.
By Claim~\ref{clm:2} we can assume that $G_{t}=\dif|_{S^{m-1}}$ for all $t\in[1,1.5]$.
Hence $g$ and $G$ yield the following $C^{\infty}$ leaf preserving diffeomorphism 
$$
T:S^{m-1}\times(0,2]\to S^{m-1}\times(0,2],
\qquad
T(y,s) = \left\{ 
\begin{array}{cc}
 g(y,s), & s\in(0,1] \\
G(y,s), & s\in[1,2].
\end{array}
\right.
$$

Notice that $T(y,2)=G(y,2)=y$.
Then by Claim~\ref{clm:2} $T$ is isotopic via a leaf preserving isotopy relatively $S^{m-1}\times(0,a]$ to a diffeomorphism $\hat T$ which is fixed on $S^{m-1}\times(c,2]$.
Denote $\hat g=\hat T|_{S^{m-1}\times(0,1]}$.
The restriction of this isotopy to $S^{m-1}\times(0,1]$ gives a leaf preserving isotopy relatively $S^{m-1}\times (0,a]$ of $g$ to a diffeomorphism $\hat g$ with desired properties.
The construction of homotopy is schematically presented in Figure~\ref{fig:homotopy}.
\end{proof}
\begin{figure}[ht]
\includegraphics[height=5cm]{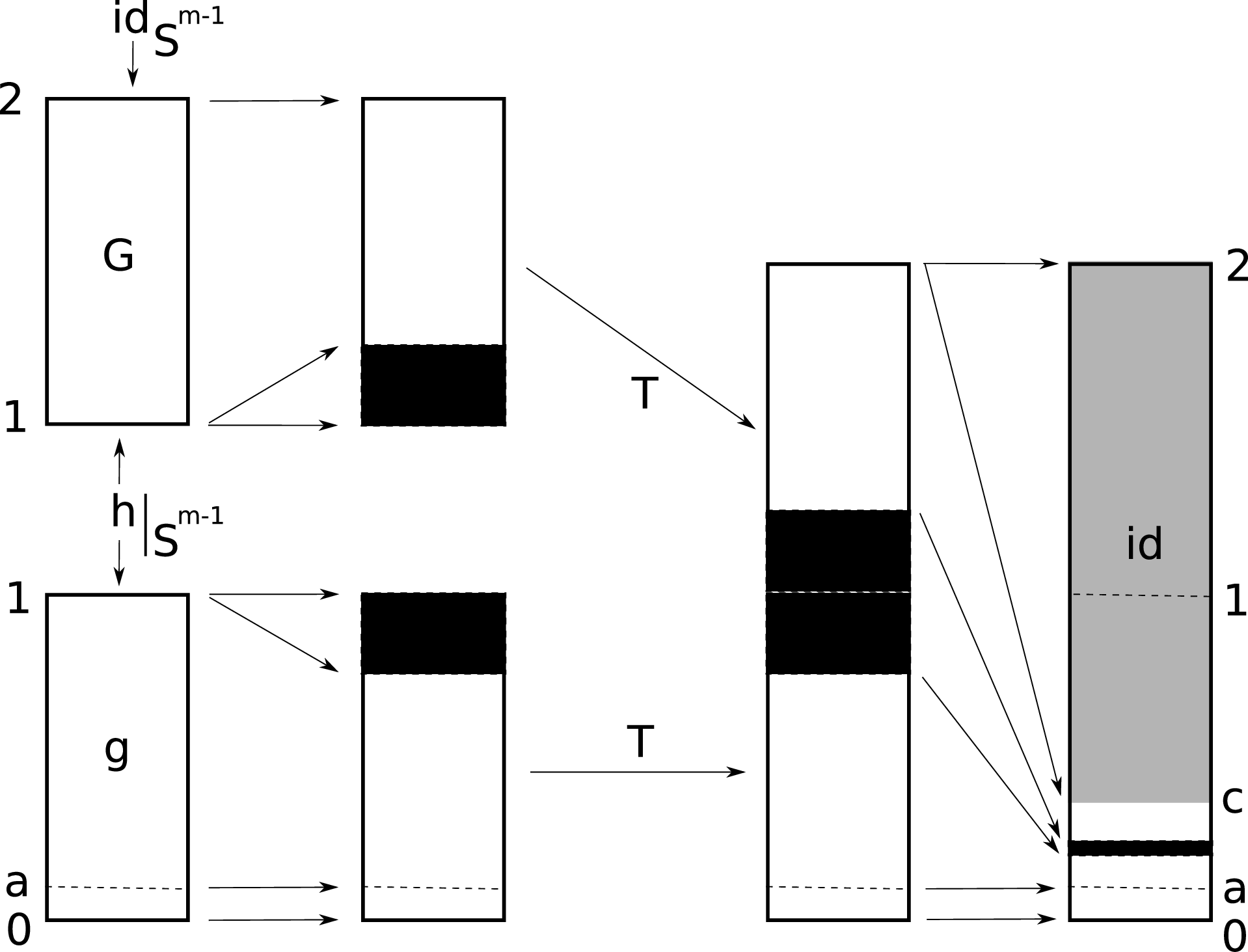}
\caption{}\protect\label{fig:homotopy}
\end{figure}

Suppose now that $m=2,3,4$.
Then every orientation preserving diffeomorphism of $S^{m-1}$ is $C^{\infty}$-isotopic to $\id_{S^{m-1}}$, whence $\Stab=\Stabfpl$, and therefore $\Stab=\StabIdf{0}=\Stabfpl$.
For $m=2$ this is rather trivial, for $m=3$ is proved by S.~Smale~\cite{Smale}, 
and for $m=4$ by A.~Hatcher~\cite{Hatcher}.

If $m\geq5$, then $\Diff^{+}(S^{m-1})$ is not connected in general, see e.g.~\cite{Novikov}, and therefore Theorem~\ref{th:main} is not applicable.
\end{proof}

\section{Linear symmetries of homogeneous polynomials}\label{sect:lin_sym}
Let $\amap:\RRR^2\to\RRR$ be a homogeneous polynomial of degree $p\geq2$ given by~\eqref{equ:homog_poly}
$$ 
\amap(x,y) = \pm \prod_{i=1}^{l} L_i^{\alpha_i}(x,y) \cdot  \prod_{j=1}^{k} Q_j^{\beta_j}(x,y).
$$
Denote 
$$\LStabf=\Stabfpl\cap\GLPR.$$
Thus $\LStabf$ consists of preserving orientation linear automorphisms $\dif:\RRR^2\to\RRR^2$ such that $\amap\circ\dif=\amap$.
Also notice that $\LStabf$ is a closed subgroup of $\GLPR$, and therefore it is a Lie group.
Denote by $\LStabf_{0}$ the connected component of the unit matrix $\id_{\RRR^2}$ in $\LStabf$.

In this section we recall the structure of $\LStabf$.
Notice that we may make linear changes of coordinates to reduce $\amap$ to a convenient form.
Then $\LStabf$ will change to a conjugate subgroup in $\GLPR$.

\begin{lemma}\label{lm:LStabf_nontriv}
If $\deg\amap$ is even, then $\amap(-z)\equiv\amap(z)$, i.e. $-\id_{\RRR^2}\in\LStabf$.
Therefore in this case $\LStabf$ is a non-trivial group.
\hfill
\qed
\end{lemma}

We will distinguish the following five cases of $\amap$.

(A) $l=1$, $k=0$, $\amap=L_1^{\alpha_1}$.
By linear change of coordinates we can assume that $L_1(x,y)=y$ and thus
$\amap(x,y)=y^{\alpha_1}$.
Then
$$
\LStab_{0}(y^{\alpha_1}) = 
\left\{ \left(\begin{smallmatrix} a & 0 \\ b & 1  \end{smallmatrix} \right) \  : \ a>0
 \right\}.
$$
If $\alpha_1$ is odd then $\LStab_{0}=\LStabf$, otherwise, $\LStabf$ consists of two connected components $\LStab_{0}$ and $-\LStab_{0}$.

(B) $l=2$, $k=0$, $\amap=L_1^{\alpha_1}\,L_2^{\alpha_2}$.
By linear change of coordinates we can assume that $L_1(x,y)=x$, $L_2(x,y)=y$ and thus
$\amap(x,y)=x^{\alpha_1}\,y^{\alpha_2}$.
Then
$$
\LStab_{0}(x^{\alpha_1}\,y^{\alpha_2}) = \left\{
\left(\begin{smallmatrix} e^{\alpha_2\,t} & 0 \\ 0 & e^{-\alpha_1\,t}\end{smallmatrix}\right) \ : \ t\in\RRR
\right\}.
$$
Moreover, $\LStabf/\LStab_{0}$ is isomorphic with some subgroup of $\ZZZ_{4}$ generated by the rotation of $\RRR^2$ by $\pi/2$.

(C) $l=0$, $k=1$, $\amap=Q_1^{\beta_1}$.
By linear change of coordinates we can assume that $Q_1(x,y)=x^2+y^2$, whence 
$\amap(x,y)=(x^2+y^2)^{\beta_1}$.
Then
$$
\LStab(x^2+y^2) = SO(2,\RRR)=\left\{
\left(\begin{smallmatrix} \cos t & \sin t \\ -\sin t & \cos t\end{smallmatrix}\right) \ : \ t\in[0,2\pi)
\right\}.
$$

The above statements are elementary and we left them for the reader.
Notice also that in the cases (A)-(C) $l+2k\leq 2$.
The remaining two cases are the following:

(D) $l=0$, $k\geq2$, $\amap=Q_1^{\beta_1}\cdots Q_k^{\beta_k}$.
In this case $\deg\amap$ is even, whence $\LStabf$ is non-trivial.

(E) $l\geq1$, $l+2k\geq3$.

\begin{lemma}\label{lm:LStabf}
In the cases {\rm(D)} and {\rm(E)} $\LStabf$ is a finite cyclic subgroup of $\GLPR$.
Moreover, in the case {\rm(E)} $\LStabf$ is a subgroup of $\ZZZ_{2l}$.
\end{lemma}
\begin{proof}
In fact the cyclicity of $\LStabf$ for the case $l+2k-1\geq 2$ can be extracted from the paper of W.\;C.\;Huffman, \cite{Huffman}, where the symmetries of complex binary forms are classified.
Regard $\amap:\CCC^2\to\CCC$ as a complex polynomial with real coefficients.
Then by~\cite{Huffman} the subgroup $\LStab_{\CCC}(\amap)$ of $GL(2,\CCC)$ consisting of \emph{complex} symmetries of $\amap$ turned out to be of one of the following types: cyclic, dihedral, tetrahedral, octahedral, and icosahedral.
Notice $\LStabf$ is the subgroup of $\LStab_{\CCC}(\amap)$ 
consisting of \emph{preserving orientation real} symmetries of $\amap$, i.e. automorphisms which also leave invariant $2$-plane $\RRR^2\subset\CCC^2$ of real coordinates and preserve its orientation.
Then it follows from the structure of symmetries of regular polyhedrons, that $\LStabf$ must be cyclic.

Nevertheless, since we need a very particular case of~\cite{Huffman} and for the sake of completeness, we will present a short elementary proof.
It suffices to show that $\LStabf$ is finite, see~\ref{clm:LStabf_finite}.
This will imply that $\LStabf$ is isomorphic with a finite subgroup of $SO(2)$, and therefore is cyclic.
Also notice that the fact that $\LStabf$ is discrete also follows from~\cite{Maks:jets}. 
First we establish the following three statements:
\begin{claim}\label{clm:Qh_tQ_t_deth}
Let $\dif\in\GLPR$ and $Q$ be a positive definite quadratic form such that $Q\circ\dif=t\,Q$ for some $t>0$.
Then $t=\det(h)$.
\end{claim}
\begin{proof}
By linear change of coordinates we can assume that $Q(z)=|z|^2$.
Then $\dif(z) = \sqrt{t} e^{i\psi}z$ for some $\psi\in\RRR$, hence $\det(\dif)=t$.
\end{proof}

\begin{claim}\label{clm:Q1Q2h}
Let $Q_1, Q_2$ be a positive definite quadratic form such that $\frac{Q_1}{Q_2}\not\equiv\mathrm{const}$.
Let also $\dif\in\GLPR$ be such that $Q_i\circ\dif=t Q_i$ for $i=1,2$, where $t=\det(h)$.
Then  $h(z)=\pm \sqrt{t}\,z$.
\end{claim}
\begin{proof}
We can assume that $Q_1(x,y)=x^2+y^2$ and $Q_2(x,y)=ax^2 + by^2$, where $a,b>0$ and either $a\not=1$ or $b\not=1$.
Denote $g(z) = \dif(z)/\sqrt{t}$.
Then $Q_i\circ g =Q_i$, i.e. $g$ preserves every circle $x^2+y^2=\mathrm{const}$ and every ellipse $ax^2 + by^2=\mathrm{const}$.
Therefore $g=\pm\id_{\RRR^2}$, and $\dif(z) = \pm\sqrt{t} z$.
\end{proof}

\begin{claim}\label{clm:t_id_in_LStabf}
If \ $t\cdot\id_{\RRR^2}\in\LStabf$ for some $t\in\RRR$, then $t=\pm1$.
\end{claim}
\begin{proof}
Let $z\in\RRR^2$ be such that $\amap(z)\not=0$.
Since $\amap$ is homogeneous, we have 
$\amap(z)=\amap(tz) = t^{\deg \amap}\amap(z),$
whence $t=\pm1$.
\end{proof}

Let $\dif\in\LStabf$.
Since $L_i$ and $Q_j$ are irreducible over $\RRR$, so are $L_i\circ \dif$ and $Q_j\circ\dif$.
Therefore the identity $\amap\circ\dif=\amap$ implies that ``$\dif$ permutes $L_i$ and $Q_j$ up to non-zero multiples''.
This means that for every $i$ there exist $i'$ and $s_i\in\RRR\setminus\{0\}$, and for every $j$ there exist $j'$ and $t_j>0$ such that 
$$
L_i(\dif(z)) = s_i\,L_{i'}(z),
\qquad
Q_j(\dif(z)) = t_j\, Q_{j'}(z).
$$
Denote by $\Sym{r}$ the group of permutations of $r$ symbols.
Then we have a well-defined homomorphism
$$
\mu:\LStabf \to \Sym{l} \times \Sym{k}
$$
associating to every $\dif\in\LStab$ its permutations of $L_i$ and $Q_j$.

\begin{claim}\label{clm:LStabf_finite}
If $l+2k\geq3$, then $\ker\mu\subset\{\pm\id_{\RRR^2}\}$, whence $\LStabf$ is a finite group.
\end{claim}
\begin{proof}
Let $\dif\in\ker\mu$.
Thus $L_i\circ\dif=s_i L_i$ and $Q_j\circ\dif=t_j Q_j$ for all $i,j$.
We will show that $\dif=t\cdot\id_{\RRR^2}$ for some $t\not=0$.
Then it will follow from Claim~\ref{clm:t_id_in_LStabf} $\dif=\pm\id$.

Notice that $\dif$ preserves every line $\{L_i=0\}$ and thus has $l$ distinct eigen directions.

a) Therefore if $l\geq3$, then $\dif=t\cdot\id_{\RRR^2}$ for some $t\in\RRR$.

b) Moreover, if $k\geq2$, then by Claim~\ref{clm:Q1Q2h}, we also have $\dif=\pm\id_{\RRR^2}$.

c) Suppose that $1\leq l \leq 2$ and $k=1$.
We can assume that $Q(z)=|z|^2$ and that $\dif(z)= t\, e^{i\psi}z$ for some $t>0$ and $\psi\in\RRR$.
Since $\dif$ has $l\geq1$ eigen directions, we obtain that $\dif(z)=\pm \sqrt{t} z$.
\end{proof}

Thus $\LStabf \approx \ZZZ_{n}$ for some $n\in\NNN$.
Let $\dif$ be a generator of $\LStabf$.
Then we can assume that $\dif(z) = e^{2\pi i/n}z$.
It remains to prove the latter statement.
\begin{claim}\label{clm:LStabf_Zl}
Suppose that $l\geq1$.
Then $n$ divides $2l$, whence $\LStabf$ is isomorphic with a subgroup of $\ZZZ_{2l}$.
\end{claim}
\begin{proof}
Since $\amap\circ\dif=\amap$, it follows that $\dif(\amap^{-1}(0))=\amap^{-1}(0)$.
By assumption $l\geq1$, whence $\amap^{-1}(0)=\mathop\cup\limits_{i=1}^{l}\{L_i=0\}$ is the union of $l$ lines passing through the origin.
This set can be viewed as the union of $2l$ rays starting at the origin, and these rays are cyclically shifted by $\dif$.
Moreover, if $\dif^{t}$ preserves at least one of these rays, then $\dif^{t}=\id_{\RRR^2}$.
Therefore $n$ divides $2l$.
\end{proof}
Lemma~\ref{lm:LStabf} is completed.
\end{proof}

\section{Proof of Theorem~\ref{th:Stabf1_not_Stabf0}.}\label{sect:proof-th:Stabf1_not_Stabf0}
Let $\amap:\RRR^2\to\RRR$ be a homogeneous polynomial of degree $p\geq1$ given by~\eqref{equ:homog_poly}
$$ 
\amap(x,y) = \pm \prod_{i=1}^{l} L_i^{\alpha_i}(x,y) \cdot  \prod_{j=1}^{k} Q_j^{\beta_j}(x,y).
$$
We will refer to the cases (A)-(E) of $\amap$ considered in the previous section.
We have to show that $\StabIdf{\infty}=\cdots=\StabIdf{1}$ and that $\StabIdf{1}\not=\StabIdf{0}$ iff $\amap$ is of the case (D).

Our first aim is to identify $\StabIdf{r}$ with the group $\DidhFld{r}$ for some vector field $\hFld$ on $\RRR^2$, see Lemma~\ref{lm:DhF_Sf}.
Then we will use the shift map of $\hFld$.
Denote 
$$
D = \pm \prod_{i=1}^{l} L_i^{\alpha_i-1} \cdot  \prod_{j=1}^{k} Q_j^{\beta_j-1}.
$$
Then 
$$
\amap=L_1 \cdots L_{l} \cdot Q_{1} \cdots Q_{q} \cdot D
$$
and it is easy to see that $D$ is the greatest common divisor of $\amap'_{x}$ and $\amap'_{y}$ in the ring $\RRR[x,y]$.

Let $\Fld=-\amap'_y\;\frac{\partial}{\partial x} \;+\; \amap'_x\;\frac{\partial}{\partial x}$ be the Hamiltonian vector field of $\amap$ on $\RRR^2$ and
$$
\hFld=\Fld/D = -(\amap'_y/D)\;\frac{\partial}{\partial x} \;+\; (\amap'_x/D)\;\frac{\partial}{\partial x}.
$$
We will call $\hFld$ the \emph{reduced\/} Hamiltonian vector field of $\amap$.
Notice that $$\deg\hFld=l+2k-1$$
and the coordinate functions of $\hFld$ are relatively prime in $\RRR[x,y]$.

As noted in Example~\ref{exmp:part-func-Fld} the singular partitions $\strucf$ and $\strucFld$ coincide.
Let us describe the singular partition $\struchFld=(\partithFld,\singhFld)$.
Recall that elements of $\partithFld$ are the orbits of $\hFld$ and $\singhFld$ consists of zeros of $\hFld$.
Consider the following cases, see Figure~\ref{fig:cases}.

(A) $\amap=y^{\alpha_1}$.
Then $D=y^{\alpha_1-1}$ and $\Fld(x,y)=\alpha_1 y^{\alpha_1-1}\frac{\partial}{\partial y}$.
Hence $\hFld(x,y)=\alpha_1\frac{\partial}{\partial y}$ is a constant vector field and the partition $\struchFld$ consists of horizontal lines $\{y=\mathrm{const}\}$ being non-singular elements of $\struchFld$.

(C) and (D) $\amap=Q_1^{\beta_1} \cdots Q_k^{\beta_k}$.
In this case $\strucFld=\struchFld$.
The origin is a unique singular element of $\struchFld$.
All other elements of $\struchFld$ are level-sets $\amap^{-1}(c)$ of $\amap$ for $c>0$.

(B) and (E) either $l=2$ and $k=0$ or $l\geq1$ and $k\geq1$.
In both cases the set of singular points of $\strucFld$ consist of the origin and the set 
$$
D^{-1}(0) = \mathop\cup\limits_{i \ : \ \alpha_i\geq2}  \{ L_i=0\}
$$
of zeros of $D$ being the union of those lines $\{L_i=0\}$ for which $L_i$ is a multiple factor of $\amap$.
Since after division of $\Fld$ by $D$ the coordinate functions of $\hFld=\Fld/D$ are relatively prime, it follows that $0$ is a unique singular element of $\struchFld$.
Hence non-singular elements of $\struchFld$ are the connected components $\amap^{-1}(c)$ for $c\not=0$ and the half-lines in $\amap^{-1}(0)\setminus0$, see Figure~\ref{fig:partit}.

\begin{center}
\begin{figure}[ht]
\begin{tabular}{ccccccc}
\includegraphics[height=1.5cm]{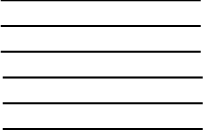} & \qquad & 
\includegraphics[height=1.5cm]{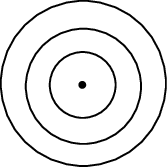} & \qquad & 
\includegraphics[height=1.5cm]{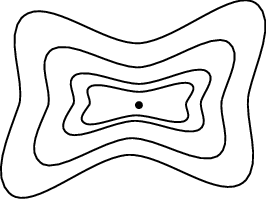} & \qquad & 
\includegraphics[height=1.5cm]{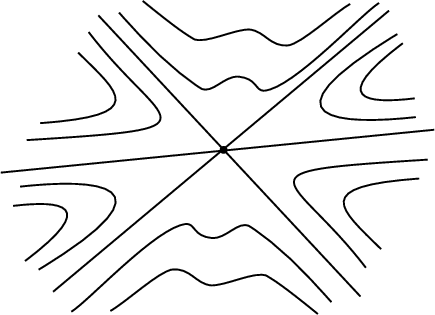}\\
Case (A) & & Case (C) & & Cases (D) & & Cases (B) and (E)
\end{tabular}
\caption{}\protect\label{fig:cases}
\end{figure}
\begin{figure}[ht]
\begin{tabular}{ccc}
\includegraphics[height=1.5cm]{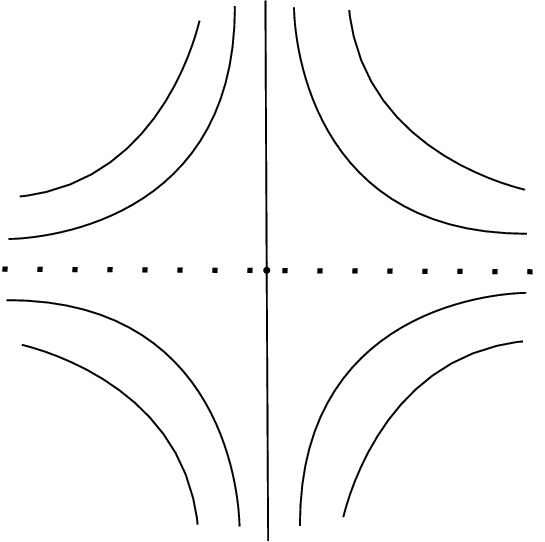} & \ \qquad \ & \includegraphics[height=1.5cm]{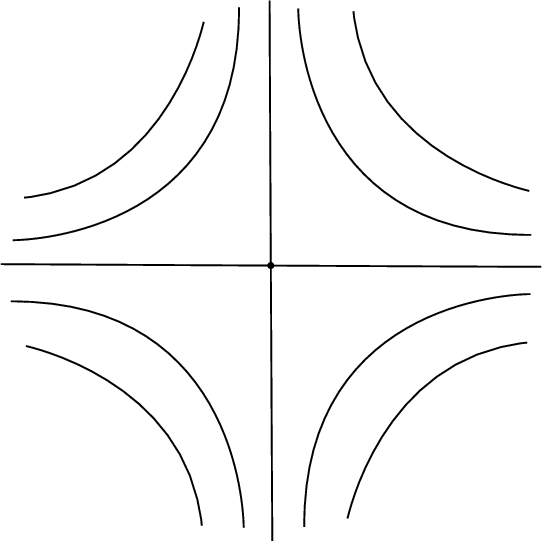} \\
a) $\Fld(x,y)=-2xy \frac{\partial}{\partial x}+y^2 \frac{\partial}{\partial y}$ & &
b) $\hFld(x,y)=-2x \frac{\partial}{\partial x}+y \frac{\partial}{\partial y}$ \\
\end{tabular}
\caption{Case (B). Hamiltonian and reduced Hamiltonian vector fields for $\amap(x,y)=x\,y^2$.}\protect\label{fig:partit}
\end{figure}
\end{center}

Since $\amap$ is constant along orbits of $\Fld$ and $\hFld$, it follows that
\begin{equation}\label{equ:Did_FG_Stabf}
\DF \;\subset\; \DhF \;\subset\; \Stabf.
\end{equation}
Also notice that $\DF$ consists of those $\dif\in\DhF$ which fixes every critical point of $\amap$.

\begin{lemma}\label{lm:DhF_Sf}
$\DidhFld{r}=\StabIdf{r}$ for all $r\in\NNi$.
\end{lemma}
\begin{proof}
It follows from~\eqref{equ:Did_FG_Stabf} that $\DidhFld{r}\subset\StabIdf{r}$.

Conversely,  let $\dif\in\StabIdf{r}$, so there exists an $r$-isotopy $\dif_t:\RRR^2\to\RRR^2$ between $\dif_0=\id_{\RRR^2}$ and $\dif_1=\dif$ in $\Stabf$, i.e. 
\begin{equation}\label{equ:fht_f}
\amap\circ\dif_t=\amap, \qquad t\in I.
\end{equation}
We claim that every $\dif_t\in\DhF$, i.e. $\dif_t(\omega)=\omega$ for every element $\omega$ of $\struchFld$.
This will mean that $\{\dif_t\}$ is an $r$-isotopy in $\DhF$, whence $\dif\in\DidhFld{r}$.

It follows from~\eqref{equ:fht_f} that $\dif_t(\amap^{-1}(c))=\amap^{-1}(c)$ for every $c\in\RRR$ and $\dif_t(\singf)=\singf$.
Since $\dif_0=\id_{\RRR^2}$ preserves every connected component $\omega$ of $\amap^{-1}(c)\setminus\singf$, so does $\dif_t$, $t\in I$.
If either $c\not=0$, or $c=0$ but $\amap$ is of either the cases (A), (C), or  (D), then by definition every such $\omega$ is an element of $\struchFld$.

Let $c=0$.
We claim that in the cases (B) and (E) $\dif_t(0)=0$ for all $t\in I$.
Indeed, in these cases the origin is ``the most degenerate point among all other points of $\amap^{-1}(0)$''.
This means the following.

For every $z\in\amap^{-1}(0)$ denote by $p_{z}$ the least number such that $p_{z}$-jet of $\amap$ at $z$ does not vanish, i.e. $j^{p_{z}-1}(\amap,z)=0$ while $j^{p_{z}}(\amap,z)\not=0$.
In other words, the Taylor series of $\amap$ at $z$ starts with terms of order $p_z$.
It is easy to see that for the origin $p_0=\deg\amap$, while for all other points $z\in\amap^{-1}(0)$ we have that $p_z<\deg\amap$.
Also notice that this number $p_z$ is preserved by any diffeomorphism $\dif\in\Stabf$, i.e. $p_{\dif(z)}=p_{z}$.
It follows that $\dif_t(0)=0$.

It remains to note that by continuity every $\dif_t$ preserves connected components of $D^{-1}(0)\setminus\{0\}$.
Hence $\dif_t\in\DhF$ for all $t\in I$.
\end{proof}

Now we can complete Theorem~\ref{th:Stabf1_not_Stabf0}.
Let $\flow$ be the local flow on $\RRR^2$ generated by $\hFld$, and $\Shift$ be the shift map of $\hFld$, see Section~\ref{sect:shift-map}.
The following statement was established in~\cite{Maks:Shifts}.
\begin{lemma}
In the cases {\rm(A)-(C)} (i.e. when $\deg\hFld\leq1$) for every $\dif\in\EidhFld{0}$ there exists a smooth function $\sigma:\RRR^2\to\RRR$ such that $\dif(z)=\flow(x,\sigma(x))$, i.e. $\im(\Shift)=\EidhFld{0}$.
\end{lemma}
It follows from Lemmas~\ref{lm:imShift_EidVFinfty} and~\ref{lm:DhF_Sf} that in the cases (A)-(C)
\begin{equation}\label{equ:Stabinf_Stab0}
\StabIdf{\infty} = \DidhFld{\infty}=\cdots= \DidhFld{0} =\StabIdf{0}.
\end{equation}

The following lemma is a consequence of results of~\cite{Maks:jets,Maks:hamv2}.
\begin{lemma}\label{lm:imShift_DE}
In the cases {\rm(D)} and {\rm(E)} (i.e. when $\deg\hFld\geq2$) $\im(\Shift)$ consists of all $\dif\in\EhF$ whose tangent map $T_0\dif:T_0\RRR^2\to T_0\RRR^2$ at $0$ is the identity. 
\end{lemma}
\begin{corollary}\label{cor:imShift_DE}
$\EidhFld{1}\subset\im(\Shift)$, whence similarly to~\eqref{equ:Stabinf_Stab0} we get $\StabIdf{\infty} = \cdots =\StabIdf{1}$.
\end{corollary}
\begin{proof}[Proof of Corollary.]
Let $\dif\in\EidhFld{1}$.
Then there exists a $1$-homo\-topy $\dif_t$ between $\dif_0=\id_{\RRR^2}$ and $\dif_1=\dif$ in $\EhF$.
In particular, $T_0\dif_t$ is continuous in $t$.

Since $\amap$ is homogeneous, it follows from~\cite[Lemma~36]{Maks:Shifts}, that $T_0\dif_t$ regarded as a linear automorphism of $\RRR^2$ also must preserve $\amap$, i.e. $T_0\dif_t\in\LStabf$.
Therefore the family of maps $T_0\dif_t$ can be regarded as a homotopy in $\LStabf$.
But by Lemma~\ref{lm:LStabf} in the cases (D) and (E) the group $\LStabf$ is discrete, whence all $T_0\dif_t$ coincide with the identity map $T_0\dif_0=\id_{\RRR^2}$.
In particular, $T_0\dif=\id_{\RRR^2}$, whence by Lemma~\ref{lm:imShift_DE} $\dif\in\im(\Shift)$.
\end{proof}

It remains to show that $\StabIdf{1}$ and $\StabIdf{0}$ coincide in the case (E) and are distinct in the case (D).

(D) Let $\amap$ be a product of at least two distinct definite quadratic forms.
Then by Theorem~\ref{th:Stabf-Bn} $\StabIdf{0}=\Stabfpl$.
Moreover, since $\deg\amap$ is even, we have that $-\id_{\RRR^2} \in \Stabfpl=\StabIdf{0}$, see Lemma~\ref{lm:LStabf_nontriv}.
On the other hand by Lemma~\ref{lm:imShift_DE} for every $\dif\in\StabIdf{1}$ its tangent map $T_0\dif=\id_{\RRR^2}  \not= -\id_{\RRR^2}$.
Hence $\StabIdf{1}\not=\StabIdf{0}$.

(E) In this case $\amap^{-1}(0)$ is a union of $l\geq1$ straight lines $\{L_i=0\}$ passing through the origin.
Let $\dif\in\DidhFld{0}$.
Since there exists a homotopy between $\dif$ and $\id_{\RRR^2}$ in $\DidhFld{0}$, it follows that $\dif$ preserves every half-line of $\amap^{-1}(0)\setminus\{0\}$.
Therefore so does $T_0\dif\in\LStabf$.
Then it follows from Claim~\ref{clm:LStabf_Zl} that $T_0\dif=\id_{\RRR^2}$, whence by Lemma~\ref{lm:imShift_DE} $\dif\in\im(\Shift)$.
Thus $\DidhFld{0}\subset\im(\Shift)$.

Now we get from Lemma~\ref{lm:imShift_EidVFinfty} that 
$$\StabIdf{1}=\DidhFld{1}=\DidhFld{0}=\StabIdf{0}.$$
Theorem~\ref{th:Stabf1_not_Stabf0} is completed.

\section{Acknowledgement}
I would like to thank V.\;V.\;Sharko and E.\;Polulyah for useful discussions.
I also thank anonymous referee for careful reading of this manuscript and critical remarks which allow to clarify many points.

\end{document}